\documentclass[reqno]{amsart}

\usepackage{amsmath,amsfonts,amssymb,amscd,verbatim,delarray,verbatim,epigraph}

\usepackage{xcolor}

\usepackage{amsthm}

\usepackage{url}

\newcommand{\F}{{\mathbb F}}

\newcommand{\Q}{{\mathbb Q}}

\begin{document}

\newtheorem{theorem}{Theorem}

\newtheorem{corollary}[theorem]{Corollary}
\newtheorem{corol}[theorem]{Corollary}
\newtheorem{conj}[theorem]{Conjecture}
\newtheorem{prop}[theorem]{Proposition}
\newtheorem{lemma}[theorem]{Lemma}

\theoremstyle{definition}
\newtheorem{definition}[theorem]{Definition}
\newtheorem{example}[theorem]{Example}

\newtheorem{remarks}[theorem]{Remarks}
\newtheorem{remark}[theorem]{Remark}

\newtheorem{question}[theorem]{Question}
\newtheorem{problem}[theorem]{Problem}

\newtheorem{quest}[theorem]{Question}
\newtheorem{questions}[theorem]{Questions}

\def\toeq{{\stackrel{\sim}{\longrightarrow}}}
\def\into{{\hookrightarrow}}


\def\alp{{\alpha}}  \def\bet{{\beta}} \def\gam{{\gamma}}
 \def\del{{\delta}}
\def\eps{{\varepsilon}}
\def\kap{{\kappa}}                   \def\Chi{\text{X}}
\def\lam{{\lambda}}
 \def\sig{{\sigma}}  \def\vphi{{\varphi}} \def\om{{\omega}}
\def\Gam{{\Gamma}}   \def\Del{{\Delta}}
\def\Sig{{\Sigma}}   \def\Om{{\Omega}}
\def\ups{{\upsilon}}


\def\F{{\mathbb{F}}}
\def\BF{{\mathbb{F}}}
\def\BN{{\mathbb{N}}}
\def\Q{{\mathbb{Q}}}
\def\Ql{{\overline{\Q }_{\ell }}}
\def\CC{{\mathbb{C}}}
\def\R{{\mathbb R}}
\def\V{{\mathbf V}}
\def\D{{\mathbf D}}
\def\BZ{{\mathbb Z}}
\def\K{{\mathbf K}}
\def\XX{\mathbf{X}^*}
\def\xx{\mathbf{X}_*}

\def\AA{\Bbb A}
\def\BA{\mathbb A}
\def\HH{\mathbb H}
\def\PP{\Bbb P}

\def\Gm{{{\mathbb G}_{\textrm{m}}}}
\def\Gmk{{{\mathbb G}_{\textrm m,k}}}
\def\GmL{{\mathbb G_{{\textrm m},L}}}
\def\Ga{{{\mathbb G}_a}}

\def\Fb{{\overline{\F }}}
\def\Kb{{\overline K}}
\def\Yb{{\overline Y}}
\def\Xb{{\overline X}}
\def\Tb{{\overline T}}
\def\Bb{{\overline B}}
\def\Gb{{\bar{G}}}
\def\Ub{{\overline U}}
\def\Vb{{\overline V}}
\def\Hb{{\bar{H}}}
\def\kb{{\bar{k}}}

\def\Th{{\hat T}}
\def\Bh{{\hat B}}
\def\Gh{{\hat G}}


\def\cC{{\mathcal C}}
\def\cU{{\mathcal U}}
\def\cP{{\mathcal P}}
\def\cV{{\mathcal V}}
\def\cS{{\mathcal S}}

\def\CG{\mathcal{G}}

\def\cF{{\mathcal {F}}}

\def\Xt{{\widetilde X}}
\def\Gt{{\widetilde G}}


\def\hh{{\mathfrak h}}
\def\lie{\mathfrak a}

\def\XX{\mathfrak X}
\def\RR{\mathfrak R}
\def\NN{\mathfrak N}

\def\minus{^{-1}}

\def\GL{\textrm{GL}}            \def\Stab{\textrm{Stab}}
\def\Gal{\textrm{Gal}}          \def\Aut{\textrm{Aut\,}}
\def\Lie{\textrm{Lie\,}}        \def\Ext{\textrm{Ext}}
\def\PSL{\textrm{PSL}}          \def\SL{\textrm{SL}} \def\SU{\textrm{SU}}
\def\loc{\textrm{loc}}
\def\coker{\textrm{coker\,}}    \def\Hom{\textrm{Hom}}
\def\im{\textrm{im\,}}           \def\int{\textrm{int}}
\def\inv{\textrm{inv}}           \def\can{\textrm{can}}
\def\id{\textrm{id}}              \def\Char{\textrm{char}}
\def\Cl{\textrm{Cl}}
\def\Sz{\textrm{Sz}}
\def\ad{\textrm{ad\,}}
\def\SU{\textrm{SU}}
\def\Sp{\textrm{Sp}}
\def\PSL{\textrm{PSL}}
\def\PSU{\textrm{PSU}}
\def\rk{\textrm{rk}}
\def\PGL{\textrm{PGL}}
\def\Ker{\textrm{Ker}}
\def\Ob{\textrm{Ob}}
\def\Var{\textrm{Var}}
\def\poSet{\textrm{poSet}}
\def\Al{\textrm{Al}}
\def\Int{\textrm{Int}}
\def\Smg{\textrm{Smg}}
\def\ISmg{\textrm{ISmg}}
\def\Ass{\textrm{Ass}}
\def\Grp{\textrm{Grp}}
\def\Com{\textrm{Com}}
\def\rank{\textrm{rank}}

\def\char{\textrm{char}}

\def\wid{\textrm{wd}}

\newcommand{\Or}{\operatorname{O}}

\def\tors{_\def{\textrm{tors}}}      \def\tor{^{\textrm{tor}}}
\def\red{^{\textrm{red}}}         \def\nt{^{\textrm{ssu}}}

\def\sss{^{\textrm{ss}}}          \def\uu{^{\textrm{u}}}
\def\mm{^{\textrm{m}}}
\def\tm{^\times}                  \def\mult{^{\textrm{mult}}}

\def\uss{^{\textrm{ssu}}}         \def\ssu{^{\textrm{ssu}}}
\def\comp{_{\textrm{c}}}
\def\ab{_{\textrm{ab}}}

\def\et{_{\textrm{\'et}}}
\def\nr{_{\textrm{nr}}}

\def\nil{_{\textrm{nil}}}
\def\sol{_{\textrm{sol}}}
\def\End{\textrm{End\,}}

\def\til{\;\widetilde{}\;}

\def\min{{}^{-1}}

\def\AGL{{\mathbb G\mathbb L}}
\def\ASL{{\mathbb S\mathbb L}}
\def\ASU{{\mathbb S\mathbb U}}
\def\AU{{\mathbb U}}

\title{Elementary equivalence of Kac-Moody groups}

\author[Morita]{Jun Morita}

\address{
Jun Morita, Division of Mathematics, Faculty of Pure and Applied Sciences,
University of Tsukuba, Tsukuba, Ibaraki 305-8571, JAPAN}
\email{morita@math.tsukuba.ac.jp}

\author[Plotkin]{Eugene Plotkin}

\address{Eugene Plotkin, Department of Mathematics,
Bar-Ilan University, 5290002 Ramat Gan, ISRAEL}
\email{plotkin@math.biu.ac.il}

\begin{abstract}
The paper is devoted to model-theoretic properties of Kac-Moody groups with  the focus on
 elementary equivalence of Kac-Moody groups. We show that elementary equivalence of (untwisted) affine Kac-Moody groups implies coincidence of their generalized Cartan matrices and the elementary equivalence of their ground fields.
 We study also the Diophantine problem in affine Kac-Moody groups.  We show that  for the loop group the Diophantine problem is polynomial time equivalent (more precisely, Karp equivalent) to the
problem in the ground ring. Finally, we show that in affine Kac-Moody groups over finite fields the Diophantine problem is undecidable.


\end{abstract}

\maketitle


\medskip


\section{Introduction. Elementary equivalence}\label{Equiv}

Importance of logical classifications of algebraic structures goes back to the famous  works of A.Tarski and A.Malcev. Here the term \textquoteleft\textquoteleft algebraic structure\textquoteright\textquoteright or just \textquoteleft\textquoteleft algebra\textquoteright\textquoteright means a  set with operations, see \cite{Ku}.  The main problem is to figure out {\it what are the algebras logically equivalent to a given one.}
We discuss this problem from the perspectives of  Kac-Moody groups.

Given an algebra $H$, its {\it elementary theory} $Th(H)$ is the set of all sentences (closed formulas) valid on $H$.

\begin{definition}\label{elem}
Two groups $H_1$ and $H_2$ are said to be {\it elementarily equivalent} if their elementary theories coincide.
\end{definition}

Very often we fix a class of algebras $\mathcal C$ and ask what are the algebras elementarily equivalent to a given algebra inside the class $\mathcal C$.

Last years a lot of attention was concentrated around the elementary equivalence questions for linear groups. It is a kind of folklore that a group elementarily equivalent to a linear group is linear. However, the question what is the group elementarily equivalent to a given linear group, is far from being trivial

Throughout the paper we will  denote the elementary equivalent objects $A_1$ and $A_2$ by $A_1\equiv A_2$ and isomorphic objects are denoted by $A_1\simeq A_2$.

Recall that the classical Keisler-Shelah's isomorphism theorem states, in particular, that two algebras $A$ and $B$  are elementary equivalent if and only if there is an ultrafilter such that
the corresponding ultrapowers of $A$ and $B$ are isomorphic  \cite{CK}, \cite{BS}.

Let $G$ be a group from a class of groups $\mathcal C$. In many cases one can fix the cardinality of groups from $\mathcal C$. We call $G$ {\it elementarily rigid}, if  for any group $H$ in  $\mathcal C$ the equivalence $G\equiv H$ implies isomorphism  $G\simeq H$.

 So, the rigidity question with respect to elementary equivalence looks as follows.

 \begin{problem}
 Let a class of algebras $\mathcal C$ and an algebra $H\in \mathcal C$ be given. Suppose that the elementary theories of algebras $H$ and $A\in \mathcal C$ coincide. Are they  {\it elementarily rigid}, that is,  are  $H$ and $A$  isomorphic?
 \end{problem}

 In general, elementary rigidity of groups is a rather rare phenomenon. However, studying the elementary properties of linear groups resulted in various examples of such kind. Historically, the first results were obtained in \cite{Mal2}, \cite{Zilber}. Recently, a series of papers on elementary equivalence of linear groups with the emphasis on Chevalley groups and arithmetic lattices has been published, see \cite{Bu}, ~\cite{ALM},~\cite{AM},~\cite{SM},~\cite{ST}. 
 Most of the results state a kind of elementary rigidity of these groups with respect to this or that class of groups  $\mathcal C$. The major example is the class of  finitely generated groups, cf. ~\cite{ST}, ~\cite{ALM}.

 In the paper we focus our attention on the similar problems of elementary equivalence for Kac-Moody groups and obtain a sort of rigidity results.

\section{Affine Kac-Moody groups}\label{km}

Kac-Moody groups can be viewed as infinite dimensional analogs of Chevalley groups (see \cite{Stein}). One can say that they are in the same relation with respect to infinite-dimensional Kac-Moody Lie algebras (see \cite{Kac},\cite{Ct}, \cite{MP}) as Chevalley groups are related to semi-simple finite-dimensional Lie algebras. It is useful to emphasize that from the geometric point of view the passage from Chevalley groups to Kac-Moody groups means the passage from linear algebraic groups to ind-algebraic groups (see \cite{Sha}). We refer to monographs \cite{Kum}, \cite{Mar} for the detailed exposition of the above said.

Let $\widetilde G(K)$ be an (untwisted) affine Kac-Moody group over the field $K$. We want to prove that if two affine Kac-Moody groups $\widetilde G_1(K_1)$ and $\widetilde G_2(K_2)$ are elementary equivalent then their generalized Cartan matrices coincide and the fields $K_1$ and $K_2$ are elementary equivalent.


Let $A$ be an $n \times n$ indecomposable generalized Cartan matrix of (untwisted) affine type. By an (untwisted) affine  Kac-Moody group we mean the value of the simply connected Tits functor \cite{Tits}, cf.\cite{PK}, \cite{KP1},\cite{KP2}, corresponding to $A$. So, one can write $\widetilde G(K)=\widetilde G_{sc}(A,K)$.   Let $R$ be the ring of  Laurent polynomials $R=K[t,t^{-1}]$, where $K$ be a field. Denote by $G_{sc}(R)\simeq G_\Phi(R)$ the simply connected Chevalley group over $R$, where $\Phi$ is a finite root system, associated with $A$. The adjoint Chevalley group is denoted by $G_{ad}^\Phi(R)$. As usual, the elementary subgroups of Chevalley groups are denoted by $E_{sc}^\Phi(R)$ and $E_{ad}^\Phi(R)$.

Denote by $\widetilde G_{ad}(A,K)$ the adjoint Kac-Moody group (\cite{MT},\cite{Garland},\cite{Tits}). Its elementary subgroup is $\widetilde E_{ad}(A,K)$.
Denote by $Z(\widetilde G_{sc}(K))$ and $Z(\widetilde E_{sc}(K))$ the centers of $\widetilde G_{sc}(K)=\widetilde G_{sc}(A,K)$ and $\widetilde E_{sc}(K)=\widetilde E_{sc}(A,K)$, respectively.


 We have the following global picture:

 $$
\begin{array}{ccccccccc}
&&&&1&&1&&\\
&&&&\uparrow&&\uparrow&&\\
&&1&\rightarrow&K^\times&\simeq&K^\times&\rightarrow&1\\
&&\uparrow&&\uparrow&&\uparrow&&\\
1 & \rightarrow & Z(\tilde{G}_{sc}(K)) & \rightarrow & \tilde{G}_{sc}(K) & \rightarrow &
H(R) & \rightarrow & 1\\
&&||&&\uparrow&&\uparrow&&\\
1 & \rightarrow & Z(\tilde{E}_{sc}(K)) & \rightarrow & \tilde{E}_{sc}(K) & \rightarrow &
G_{ad}^\Phi(R) & \rightarrow & 1\\
&&\uparrow&&\uparrow&&\uparrow&&\\
&&1&&1&&1&&
\end{array}
$$
\noindent
Here, by definition $H(R)$ is $\widetilde G_{ad}(A,K)$ and:
$$\tilde{G}_{sc}(K)/Z(\tilde{G}_{sc}(K)) \simeq H(R) := \left\langle
G_{ad}^\Phi(R),\ \zeta(u) \mid u \in K^\times
\right\rangle
\subset {\rm Aut}_K(\mathfrak{g}(\dot{A}) \otimes R),
$$


\noindent
where $G_{ad}^\Phi(R) = G_{sc}^\Phi(R)/Z(G_{sc}^\Phi(R))$ is the corresponding adjoint Chevalley group, and
$\zeta(u) \in {\rm Aut}(R)$ is defined by $t^{\pm 1} \mapsto u^{\pm 1} t^{\pm 1}$.
Also here,
$\mathfrak{g}(\dot{A}) \otimes R$
is the corresponding loop algebra, and $\dot{A}$ is a Cartan matrix associated with $A$.

Then, there is a surjective homomorphism $\pi: \widetilde G_{sc}(A,K)\to G_\Phi(R)$. Its kernel $Ker \pi$ is isomorphic to $K^\times $. We have, $Z(\widetilde G_{sc}(A,K))=Z(\widetilde E_{sc}(A,K))$ and

$$
1 \to Z(\widetilde{E}_{sc}(A,K)) \to \widetilde{E}_{sc}(A,K)\to G_{ad}^\Phi(R) \to 1.
$$

 Then,

$$
\widetilde E_{sc}(A,K)/Z(\widetilde E_{sc}(A,K))\simeq  G_{ad}^\Phi(R)\simeq  E_{ad}^\Phi(R).
$$

\noindent
The group $G_{ad}^\Phi(R)= E_{ad}^\Phi(R)=E_{ad}^\Phi(K[t,t^{-1}])$ is used to be called loop group \cite{Garland}. The following commutator formulas are useful. Assume $|K|\geq 4$. Then
$$
[\widetilde G_{sc}(A,K), \widetilde G_{sc}(A,K)]=\widetilde E_{sc}(A,K), \quad [\widetilde E_{sc}(A,K), \widetilde E_{sc}(A,K)]=\widetilde E_{sc}(A,K).
$$
In the similar manner
$$
[\widetilde G_{ad}(A,K), \widetilde G_{ad}(A,K)]=\widetilde E_{ad}(A,K), \quad [\widetilde E_{ad}(A,K), \widetilde E_{ad}(A,K)]=\widetilde E_{ad}(A,K),
$$

\noindent
and $\widetilde E_{ad}(A,K)$ is isomorphic to the loop group.

Given a group $G$ a subgroup $H$ is called definable if there is a first order formula $\phi$ such that $a\in H$, where $a\in G$ if and only if $a$ satisfies a formula $\phi$.


\begin{theorem}\label{th1}
Let  $\widetilde E_{1_{sc}}(A_1,K_1)$ and $\widetilde E_{2_{sc}}(A_2,K_2)$ be two elementary (simply connected) affine Kac-Moody groups whose ranks are bigger than 2.
 Suppose that $\widetilde E_{1_{sc}}(A_1,K_1)\equiv \widetilde E_{2_{sc}}(A_2,K_2)$. Then $A_1=A_2$  and $K_1\equiv K_2$.
\end{theorem}


\begin{proof}

Suppose that $\widetilde E_{1_{sc}}(A_1,K_1)\equiv \widetilde E_{2_{sc}}(A_2,K_2)$. 
Since center is a definable subgroup the group $
\widetilde E_{sc}(A,K)/Z(\widetilde E_{sc}(A,K))$ is interpretable in  $
\widetilde E_{sc}(A,K)$ (see, for example, \cite{KMS}, page 131). Hence we have
$$
\widetilde E_{1_{sc}}(A_1,K_1)/Z(\widetilde E_{sc}(A_1,K_1))\equiv \widetilde E_{2_{sc}}(A_2,K_2)/Z(\widetilde E_{sc}(A_2,K_2))
$$ since both groups are interpretable in the elementary equivalent groups by the same first order formulas.
That is $\widetilde E_{1_{ad}}(A_1,K_1)\equiv \widetilde E_{2_{ad}}(A_2,K_2)$ and according to the isomorphism with Chevalley groups $E_{1_{ad}}^\Phi(R_1)\equiv  E_{2_{ad}}^\Psi(R_2)$. Recall that $R_i$, $i=1,2$ are the rings of Laurent polynomials $K_i[t,t^{-1}]$, respectively, while $\Phi$ and $\Psi$ are the root systems associated with $A_1$ and $A_2$, respectively.


Hence,
$$
E^\Phi_{ad}(K_1[t, t^{-1}])\equiv E^\Psi_{ad}(K_2[t, t^{-1}]).
$$







So, by Bunina (see \cite{Bu}, Theorem 11) we have $\Phi=\Psi$ and
$$
K_1[t,t^{-1}]\equiv K_2[t,t^{-1}].
$$

The Laurent polynomial ring $K[t,t^{-1}]$  is isomorphic to the group ring of the group $\mathbb Z$ over the field $K$. Then by Theorem 2 of \cite{KM},
$$
K_1\equiv K_2.
$$
Hence, $\Phi=\Psi$, $K_1\equiv K_2$ as required.


\end{proof}

The converse statement to Theorem \ref{th1} is not true.

\begin{prop} There exist fields $K_1$ and $K_2$ such that $K_1\equiv K_2$, but the affine Kac-Moody groups $\widetilde E_{1_{sc}}(A, K_1)$ and $\widetilde E_{2_{sc}}(A, K_2)$ are not elementary equivalent.
\end{prop}

\begin{proof}

Let  $K_1=\mathbb{Q}$ be the field of rational numbers and let $K_2$ be some its ultrapower. So these fields are elementary equivalent, i.e., $K_1\equiv K_2$. However, the polynomial rings $K_1[t]$ and $K_2[t]$ are not elementary equivalent (see Corollary 3.8 in \cite{JL}). Then the rings of
 Laurent polynomials $K_1[t,t^{-1}]$ and
  $K_2[t,t^{-1}]$ are not elementary equivalent as well, (see Lemma 2,  \cite{KM}). Then, $E_\Phi(K_1[t, t^{-1}])$ is not elementarily equivalent to  $E_\Phi(K_2[t, t^{-1}])$ by \cite{Bu}. Since $\widetilde E_{ad}(A,K_i)\simeq E^\Phi_{ad}(K_i[t, t^{-1}])$ the Kac-Moody groups $\widetilde E_{1_{sc}}(A,K_1)$ and $\widetilde E_{2_{sc}}(A, K_2)$ are not elementarily equivalent.

\end{proof}
\begin{prop}\label{propag} Let $K_1$ and $K_2$ be algebraically closed fields of characteristic zero. Then the
loop groups    $\widetilde E_{1_{ad}}(A_1, K_1)$ and $\widetilde E_{2_{ad}}(A_2,K_2)$ of rank bigger than one  are  elementarily equivalent if and only if $A_1=A_2$, $tr.deg_\mathbb{Q}(K_1)$ and $tr.deg_\mathbb{Q}(K_2)$  either both infinite or both finite and equal. Here  $tr.deg_\mathbb{Q}(K)$ stands for the transcendence degree of the field $K$ over $\mathbb{Q}$.
\end{prop}

Since algebraically closed fields are elementarily equivalent if and only if they have the same characteristic this means that for elementarily equivalent zero characteristic fields  $\widetilde E_1(A_1,K_1)\equiv\widetilde E_2(A_2,K_2)$ if and only if $A_1=A_2$, $tr.deg_\mathbb{Q}(K_1)$ and $tr.deg_\mathbb{Q}(K_2)$  either both infinite or both finite and $tr.deg_\mathbb{Q}(K_1)=tr.deg_\mathbb{Q}(K_2)$.

  The proof of Proposition \ref{propag} follows from the proof of Theorem \ref{th1} and the following

\begin{lemma} Let $K_1$ and $K_2$ be algebraically closed fields of characteristic zero. Then $K_1[t, t^{-1}]\equiv K_2[t, t^{-1}]$ if and only if   $tr.deg_\mathbb{Q}(K_1)$ and $tr.deg_\mathbb{Q}(K_2)$  either both infinite or both finite and $tr.deg_\mathbb{Q}(K_1)=tr.deg_\mathbb{Q}(K_2)$.

\end{lemma}
\begin{proof}

Let $K_1[t, t^{-1}]\equiv K_2[t, t^{-1}]$.  Then $K_1[t]\equiv K_2[t]$ (see Lemma 2,  \cite{KM}). Then $tr.deg_\mathbb{Q}(K_1)$ and $tr.deg_\mathbb{Q}(K_2)$  either both infinite or both finite and $tr.deg_\mathbb{Q}(K_1)=tr.deg_\mathbb{Q}(K_2)$, (see Corollary 3.11, \cite{JL}). Note, that this direction is valid for arbitrary infinite fields.

Let $tr.deg_\mathbb{Q}(K_1)$ and $tr.deg_\mathbb{Q}(K_2)$  either both infinite or both finite and
 $$tr.deg_\mathbb{Q}(K_1)=tr.deg_\mathbb{Q}(K_2).$$
  Then Proposition 3.19 of \cite{JL} remains true for the rings of  Laurent polynomials and implies $K_1[t, t^{-1}]\equiv K_2[t, t^{-1}]$.

\end{proof}

\begin{question} Does Proposition \ref{propag} still valid if we replace  $\widetilde E_{i_{ad}}(A_i, K_i)$ by
$\widetilde E_{i_{sc}}(A_i, K_i)$, in it, $1 \leq i\leq 2$?

\end{question}



\section{Diofantine Problem for Affine Kac-Moody groups}\label{km}

In this section we mostly follow the philosophy described in the paper \cite{Myasnikov-Sohrabi2}. The \emph{Diophantine problem} (also called the \emph{Hilbert's tenth problem} or
the \emph{generalized Hilbert's tenth problem}) in a countable algebraic structure~$\mathcal A$,
denoted $\mathcal D(\mathcal A)$, asks whether there exists an algorithm that, given a finite system
$S$ of equations in finitely many variables and coefficients in~$\mathcal A$, determines if $S$
has a solution in~$\mathcal A$ or not. In particular, if $R$ is a countable ring then $\mathcal D(R)$ asks
whether the question if a finite system of polynomial equations with coefficients
in~$R$ has a solution in~$R$ is decidable or not.

 By definition the Diophantine problem in a structure~$\mathcal A$ \emph{reduces to} the
Diophantine problem in a structure~$\mathcal B$, symbolically $\mathcal D(\mathcal A)\leqslant \mathcal D(\mathcal B)$, if there is an algorithm that for a given finite system of equations $S$ with coefficients in~$\mathcal A$ constructs a system of equations $S^*$ with coefficients in~$\mathcal B$ such that $S$ has a solution in~$\mathcal A$
 if and only if $S^*$ has a solution in~$\mathcal B$. If the reducing algorithm is polynomial-time
then the reduction is termed polynomial-time (or Karp reduction). If our structures in question are uncountable
one needs to restrict the Diophantine problems  to equations with coefficients
from a fixed countable subsets of $\mathcal A$ and $\mathcal B$, see \cite{Myasnikov-Sohrabi2} for details.

\begin{definition}
 A subset (in particular a subgroup) $H$ of a group~$G$ is Diophantine in~$G$ if it is definable in~$G$ by a formula of the type
$$
\Phi(x)= \exists y_1 \dots \exists y_n \left( \bigwedge_{i=1}^k w_i(x,y_1,\dots, y_n)=1\right),
$$
 where $w_i(x,y_1,\dots, y_n)$ is a group word on $x; y_1; \dots ; y_n$.
  \end{definition}

  Such formulas are called \emph{Diophantine} (in number theory) or \emph{positive-primitive} (in model theory).  Following~\cite{M34}, we say that

  \begin{definition}
     A structure $\mathcal A$ is \emph{e-interpretable} (or \emph{interpretable by equations}, or \emph{Diophantine interpretable}) in a structure~$\mathcal B$ if $\mathcal A$ is interpretable  in~$\mathcal B$ by Diophantine formulas.
     \end{definition}

     The main point of this definition is that if $\mathcal A$ is
e-interpretable in~$\mathcal B$ then the Diophantine problem in~$\mathcal A$ reduces in polynomial
time (\emph{Karp reduces}) to the Diophantine problem in~$\mathcal B$. Hence, if the Diophantine problem in~$\mathcal B$ is decidable, then the Diophantine problem in~$\mathcal A$ should be decidable as well. This is a traditional way to prove that some Diophantine problem is undecidable.

Originally, Hilbert formulated the Diophantine problem for the ring of integers $\mathbb Z$. It was solved in the negative by Matiyasevich \cite{M52} building
on the work of Davis, Putnam, and Robinson \cite{M17}. Subsequently, the Diophantine problem has
been studied in a wide variety of commutative rings $R$, where it was shown to be undecidable by
reducing $\mathcal D(\mathbb Z)$ to $\mathcal D(R)$.  For
further information on the Diophantine problem in different rings and fields of number-theoretic
flavour see \cite{M60}, \cite{M59}.

Let, as usual, $\widetilde G_{sc}(A,K)$ be an affine Kac-Moody group, and $G_{ad}^\Phi(R)\simeq  E_{ad}^\Phi(R)=E_{ad}^\Phi(K[t,t^{-1}])$ be the corresponding loop group. This is a Chevalley group over the ring of Laurent polynomials  and our first goal is to study the Diophantine problem for this group.

 We shall show that  the
ring $R=K[t,t^{-1}]$ is \emph{e-interpretable} in the loop group $G_{ad}^\Phi(K[t,t^{-1}])$. Let $X_\alpha$ be a one-parametric subgroup generated by elementary unipotents $x_\alpha (t)$, where $ \alpha \in \Phi, t\in R$. Then

\begin{theorem}\label{theorM6.1}
 Ring $R=K[t,t^{-1}]$ is
e-interpretable in the group ~$G_{ad}^\Phi(K[t,t^{-1}])$, rank $\Phi\geq 2$, on every one-parametric subgroup $X_\alpha$, $\alpha \in \Phi$.
\end{theorem}
\begin{proof} The proof follows from Proposition 3, Proposition 4 and Theorem 4 from \cite{BMP}.
\end{proof}

\begin{corollary}\label{corM6.1}
Field $K$ is
e-interpretable in the group ~$G_{ad}^\Phi(K[t,t^{-1}])$, rank $\Phi\geq 2$, on every one-parametric subgroup $X_\alpha$, $\alpha \in \Phi$.
\end{corollary}
\begin{proof}

By Lemma 5 in \cite{KM} (see also \cite{KM1}, Lemma 8) the field $K$ is Diophantine in $K[t,t^{-1}]$. Then the result follows from Theorem \ref{theorM6.1} and the transitivity property of e-interpretability, see Proposition 2.4 of \cite{GMO1}.
\end{proof}

\begin{corollary}\label{cor3}
Let $G=\widetilde E_{sc}(A,K)$ be an elementary affine Kac-Moody group. Let $K$  be either a finite field $\mathbb F_q$ and rank of  irreducible affine root system is $\ge 2$ or $K$ be a Dedekind ring of the arithmetic type and rank $\ge 3$. Then $K$ is e-interpretable in $G=\widetilde E_{sc}(A,K)$.
\end{corollary}

\begin{proof} It is known that for $K$ specified in the conditions of the theorem the group $G=\widetilde E_{sc}(A,K)$ is finitely generated, see \cite{AG} and \cite{Re} for finite fields and \cite{Al} for Dedekind rings of arithmetic type. Let $a_1,\ldots,a_k$ be generators of $G$. Then the center $Z(G)$ of $G$ is
 e-defined in $G$ by the system of equations $[x, a_i]=1$, where $i=1, \ldots k$. So $Z(G)$ is a Diophantine normal subgroup in $G$. Then $G/Z(G)\simeq  E_{ad}^\Phi(R)=E_{ad}^\Phi(K[t,t^{-1}])$ is e-interpretable in $G$ by Lemma 2.7 from \cite{GMO1}. In view of Corollary\ref{corM6.1} $K$ is e-interpretable in  $E_{ad}^\Phi(K[t,t^{-1}])$. Once again by transitivity property  $K$ is e-interpretable in $G=\widetilde E_{sc}(A,K)$.

\end{proof}

\begin{corollary}\label{cor4}
Let $G=\widetilde G_{sc}(A,K)$ be an affine Kac-Moody group. Let $K=\mathbb F_q$, $|K|\geq 4$,   be a finite field and rank of  irreducible affine root system is $\ge 2$. Then $K$ is e-interpretable in $G=\widetilde G_{sc}(A,K)$.
\end{corollary}
\begin{proof}

Since  $|K|\geq 4$, then
$$
 [\widetilde E_{sc}(A,K), \widetilde E_{sc}(A,K)]=\widetilde E_{sc}(A,K).
$$

Moreover, $\widetilde E_{sc}(A,K)$ has finite commutator width, see Theorem D, \cite{KPV}. Hence, $\widetilde E_{sc}(A,K)$ is Diophantine in $\widetilde G_{sc}(A,K)$, see Section 2.3, \cite{GMO1}. So, $\widetilde E_{sc}(A,K)$ is e-interpretable in $\widetilde G_{sc}(A,K)$.  By Corollary \ref{cor3} the field $K$ is e-interpretable in $\widetilde E_{sc}(A,K)$. Hence $K$ is e-interpretable in $\widetilde G_{sc}(A,K)$.
\end{proof}
Note that in a similar way $K[t,t^{-1}]$ is also e-interpretable in $\widetilde G_{sc}(A,K)$.

In fact, the proof of Theorem \ref{theorM6.1} consists of two parts.  Modification of the Double Centralizer Theorem from \cite{ST} ( Theorem 1.6) and its more general variant Theorem 3 from \cite{BMP} imply that all one-parametric subgroups are defined by Diophantine formulas, and hence are Diophantine sets. It remains to interpret ring operations on $R$  inside the one-parametric subgroups $X_\alpha$. We  show how to do that for some generic case in the way most appropriate for generalizations in Kac-Moody groups.

Assume that $\Phi$ contains a subsystem $A_2$. This case can be considered as a testing one for further generalizations. Let $\alpha$ and $\beta$ be simple roots of $A_2$.

We e-interpret $R$ on $X_{\alpha+\beta}$ turning it into a ring $\langle X_{\alpha+\beta},\oplus, \otimes\rangle$
as follows.

For  $x,y\in X_{\alpha+\beta}$ we define
$$
x\oplus y=x\cdot y.
$$
Note that if $y=x_{\alpha+\beta}(a)$, $y=x_{\alpha+\beta}(b)$, then $xy=x_{\alpha+\beta}(a+b)$, which corresponds to the addition in~$R$.
To define $x\otimes y$ for given $x,y\in X_{\alpha+\beta}$  we need some notation. Let $x_1,y_1\in G$ be
such that
$$
x_1\in X_\alpha\text{ and }[x_1, x_\beta(1)] = x;\quad  y_1\in X_\beta\text{ and }[x_\alpha(1),y_1] = y.
$$
Note that such $x_1,y_1$ always exist and unique, namely if $x=x_{\alpha+\beta}(a)$, $y=x_{\alpha+\beta}(b)$, then
$x_1=x_\alpha(a)$,  $y_1=x_\beta(b)$. Define
$$
x\otimes y:= [x_1,y_1].
$$
Observe, that in this case
$$
[x_1, y_1]= [x_\alpha(a),x_\beta(b)]=x_{\alpha+\beta}(ab),
$$
so it corresponds to the multiplication in~$R$.

The map $a\mapsto x_{\alpha+\beta}(a)$ gives rise to a ring isomorphism $R\to
\langle X_{\alpha+\beta},\oplus, \otimes\rangle$. The ring $ \langle X_{\alpha+\beta},\oplus, \otimes\rangle$
 is e-interpretable in~$G$. Indeed, $X_{\alpha+\beta}$ is a Diophantine set, as it was proved above. The defined addition is clearly Diophantine in~$G$. Since we are in $A_2$ case the subgroups
$X_\alpha$ and $X_\beta$ are Diophantine as well and the multiplication $\otimes$ is also Diophantine.


All remaining cases are more technical, but the proof of Definability for them is of the same flavor, see \cite{BMP}.

\begin{conj}\label{arb}Suppose that $\widetilde G(K)=\widetilde G_{sc}(A,K)$ is an arbitrary (not necessarily affine) Kac-Moody group over a ring $K$. Then, the ring $K$ is e-interpretable in the group $\widetilde G(K)$.
\end{conj}

To prove Conjecture \ref{arb} one needs to establish a version of Double Centralizer Theorem (or some other way of  description of real root subgroups) and to use more complicated than in Chevalley case commutator relations described in, for example,  \cite{CKMS}, \cite{BP}.

On the  other hand, since the group $G_{ad}^\Phi(K[t,t^{-1}])$ is defined by a system of  polynomial equations with integer coefficients it is easy to see that this group is e-interpreted in the ring $K[t,t^{-1}]$, cf., Proposition 6 in \cite{BMP}.
Hence, for loop groups $G=G_{ad}^\Phi(K[t,t^{-1}])$, $rk \Phi \geq 2$ the Diophantine problem in any group $G$  is Karp equivalent
to the Diophantine problem in the ring R, cf., Theorem 6 in \cite{BMP}.

\begin{prop} The Diophantine problem in the affine Kac-Moody group  $G=\widetilde G_{sc}(A,\mathbb F_q)$, $rk \geq 2$ is undecidable.
\end{prop}
\begin{proof}

Let $\mathcal D(G)$ be decidable. The ring $\mathbb F_q[t,t^{-1}]$ is e-interpretable in $G=\widetilde G_{sc}(A,\mathbb F_q)$. So, if the Diophantine problem in $G=\widetilde G_{sc}(A,\mathbb F_q)$ would be decidable then the Diophantine problem in $\mathbb F_q[t,t^{-1}]$ should be decidable as well. But this is not the case, see \cite{Pappas}, \cite{Phei}.

\end{proof}

{\bf Acknowledgements.} J. Morita was partially supported by the Grants-in-Aid
for Scientific Research of Japan (Grant No. 26400005). E.Plotkin  was supported by  ISF grants 1994/20,  1623/16  and the Emmy Noether Research Institute for Mathematics. We are very grateful to B.Kunyavskii and A.Miasnikov for careful reading of the final version of the manuscript and  important remarks.  Our special thanks go to anonymous referee for the numerous
valuable corrections.

\section{Data availability}

No datasets were generated or analyzed during the current study.
The research has a purely theoretical character related to algebra and model theory.

\section{Conflict of interest}

 On behalf of all authors, the corresponding author states that there is no conflict of interest.

\end{document}